\documentclass[11pt]{amsart}
\usepackage{multicol}
\usepackage{amsthm}
\usepackage{cite}
\usepackage{array}
\usepackage{amsmath}
\usepackage{enumerate}
\usepackage{tikz}
\usetikzlibrary{calc}
\usepackage{amssymb}
\usepackage{latexsym}
\usepackage{amsfonts}
\usepackage{color}
\usepackage{mathrsfs}
\usepackage{epsfig,helvet}

\theoremstyle{plain} \numberwithin{equation}{section}
\newtheorem{thm}{Theorem}[section]
\newtheorem{theorem}[thm]{Theorem}
\newtheorem{lemma}[thm]{Lemma}
\newtheorem{corollary}[thm]{Corollary}
\newtheorem{example}[thm]{Example}
\newtheorem{definition}[thm]{Definition}
\newtheorem{proposition}[thm]{Proposition}
\newtheorem{remark}{Remark}
\begin{document}
\setcounter{page}{1}

\title[ Hasan and Padhan]{Capability of Nilpotent Lie Superalgebras of Small Dimension}

\author[Padhan]{Rudra Narayan Padhan}
\address{Centre for Data Science, Institute of Technical Education and Research  \\
	Siksha `O' Anusandhan (A Deemed to be University)\\
	Bhubaneswar-751030 \\
	Odisha, India}
\email{rudra.padhan6@gmail.com, rudranarayanpadhan@soa.ac.in}

\author{IBRAHEM YAKZAN HASAN}
\address{Centre for Data Science, Institute of Technical Education and Research \\
	Siksha `O' Anusandhan (A Deemed to be University)\\
	Bhubaneswar-751030 \\
	Odisha, India}
\email{ibrahemhasan898@gmail.com}

\author[Nayak]{Saudamini Nayak}
\address{Department of Mathematics, National Institute of Technology Calicut, 
         NIT Campus, Kozhikode-673601, 
          India\\
                India}
\email{anumama.nayak07@gmail.com, saudamini@nitc.ac.in}

\maketitle

\begin{abstract}

In this paper, we define partially capable Lie superalgebra. As an application we classify all capable nilpotent Lie superalgebras of dimension less than equal to five.

\end{abstract}

\section{Introduction}

Lie superalgebras have applications in many areas of Mathematics and Theoretical Physics as they can be used to describe supersymmetry. Kac \cite{Kac1977} gives a comprehensive description of mathematical theory of Lie superalgebras, and establishes the classification of all finite dimensional simple Lie superalgebras over an algebraically closed field of characteristic zero. In the last few years the theory of Lie superalgebras has evolved remarkably, obtaining many results in representation theory and classification. Most of the results are extension of well known facts of Lie algebras. But the classification of all finite dimensional nilpotent Lie superalgebras  upto isomorphism is still an open problem like that of finite dimensional nilpotent Lie algebras.

\smallskip

In 1904, I. Schur introduced the Schur multiplier and cover of a group in his work on projective representation \cite{Kar1987}. Batten \cite{Batten1993} introduced and studied Schur multiplier and cover of a Lie algebra and later on, studied by several authors \cite{BS1996, Batten1996}. For a finite dimensional Lie algebra $L$ over a field $\mathbb{F}$ the free presentation of $L$ is the exact sequence, $0\longrightarrow R \longrightarrow F\longrightarrow L$, where $F$ is a free Lie algebra and $R$ is an ideal of $F$. Then the Schur multiplier $\mathcal{M}(L)$ is isomorphic to $F' \cap R/[F,R]$. The notion of multiplier and cover for Lie algebras is extended to the case of Lie superalgebras and studied in \cite{
Nayak2018, NPP, SN, SN2018b, p52, p53 }. 

 \smallskip
 
 Baer \cite{Baer1938} defined the notion of capable group. Beyl et. al.,  \cite{Beyl1979} introduced the epicenter $Z^{*}(G)$ of a group $G$, and they proved that a group $G$ is capable if and only if $Z^{*}(G)=1$. Exterior center $Z^{\wedge}(G)$ of a group $G$ is defined as $Z^{\wedge}(G)=\{g \in G|~g\wedge h=1,~ \forall~h \in G\}$. It was studied for the first time in \cite{Brown2010}, and Ellis in \cite{Ellis1995} proved that $Z^{\wedge}(G)=Z^{*}(G)$.  A Lie algebra $L$ is capable if and only if $L \cong \frac{H}{Z(H)}$ for some Lie algebra $H$. The notion of epicenter $Z^{*}(L)$ of a Lie algebra $L$ is given by Alamian et. al., \cite{Alamian2008}. The non-abelian tensor product, and exterior product of Lie algebras are defined, and some of the properties are studied by Ellis \cite{Ellis1987, Ellis1991, Ellis1995}. Recently, Niroomand et. al., \cite{PMF2013} investigated the connection between epicenter and exterior center of a finite dimensional Lie algebra and have classified all capable Heisenberg Lie algebras. As an application they have shown that there exists at least one capable Lie algebra of arbitrary corank.
 \smallskip

A complete classification of nilpotent Lie superalgebras $L$ with $\dim L \leq 5$ over real and complex fields are known \cite{alv,back, Hegazi99, na}. All Lie superalgebras up to dimension four over real field are known since 1978 \cite{back}, and in  this list there are no simple Lie superalgebras. Then nilpotent Lie superalgebras up to dimension five over both real and complex fields are classified by Hegazi \cite{Hegazi99}. Further, classification of all five dimensional Lie supealgebras are done over complex field \cite{na}. Recently Isabbel et. al., have given classification of nilpotent Lie superalgebras $L$ with $\dim L \leq 5$ \cite{alv}, which includes some missing five dimensional superalgebras in \cite{Hegazi99, na}. Capable Lie superalgebras are defined and studied by Padhan et. al. \cite{Padhandetec}. In this paper we find out all capable nilpotent Lie superalgebras  over $\mathbb{C}$ whose dimension are less than equal to five.

\section{Preliminaries}

Let $\mathbb{Z}_{2}=\{\bar{0}, \bar{1}\}$ be a field. A $\mathbb{Z}_{2}$-graded vector space $V$ is simply a direct sum of vector spaces $V_{\bar{0}}$ and $V_{\bar{1}}$, i.e., $V = V_{\bar{0}} \oplus V_{\bar{1}}$. It is also referred as a superspace. We consider all vector superspaces and superalgebras are over field $\mathbb{F}$ (characteristic of $\mathbb{F} \neq 2,3$). Elements in $V_{\bar{0}}$ (resp. $V_{\bar{1}}$) are called even (resp. odd) elements. Non-zero elements of $V_{\bar{0}} \cup V_{\bar{1}}$ are called homogeneous elements. For a homogeneous element $v \in V_{\sigma}$, with $\sigma \in \mathbb{Z}_{2}$ we set $|v| = \sigma$ as the degree of $v$. A  subsuperspace (or, subspace)  $U$ of $V$ is a $\mathbb{Z}_2$-graded vector subspace where  $U= (V_{\bar{0}} \cap U) \oplus (V_{\bar{1}} \cap U)$. We adopt the convention that whenever the degree function appears in a formula, the corresponding elements are supposed to be homogeneous. 

\smallskip 

A  Lie superalgebra (see \cite{Kac1977, Musson2012,GKL2015,YAM2000}) is a superspace $L = L_{\bar{0}} \oplus L_{\bar{1}}$ with a bilinear mapping $ [., .] : L \times L \rightarrow L$ satisfying the following identities:
\begin{enumerate}
\item $[L_{\alpha}, L_{\beta}] \subset L_{\alpha+\beta}$, for $\alpha, \beta \in \mathbb{Z}_{2}$ ($\mathbb{Z}_{2}$-grading),
\item $[x, y] = -(-1)^{|x||y|} [y, x]$ (graded skew-symmetry),
\item $(-1)^{|x||z|} [x,[y, z]] + (-1)^{ |y| |x|} [y, [z, x]] + (-1)^{|z| |y|}[z,[ x, y]] = 0$ (graded Jacobi identity),
\end{enumerate}
for all $x, y, z \in L$. Clearly $L_{\bar{0}}$ is a Lie algebra, and $L_{\bar{1}}$ is a $L_{\bar{0}}$-module. If $L_{\bar{1}} = 0$, then $L$ is just Lie algebra, but in general, a Lie superalgebra is not a Lie algebra.  A Lie superalgebra $L$ is called abelian if  $[x, y] = 0$ for all $x, y \in L$. Lie superalgebras without even part, i.e., $L_{\bar{0}} = 0$, are  abelian. A subsuperalgebra (or subalgebra) of $L$ is a $\mathbb{Z}_{2}$-graded vector subspace that is closed under bracket operation. The graded subalgebra $[L, L]$, of $L$, is known as the derived subalgebra of $L$. A $\mathbb{Z}_{2}$-graded subspace $I$ is a graded ideal of $L$ if $[I, L]\subseteq I$. Consider
$Z(L) = \{z\in L : [z, x] = 0\;\mbox{for all}\;x\in L\},$
is a graded ideal and it is called the {\it center} of $L$. A homomorphism between superspaces $f: V \rightarrow W $ of degree $|f|\in \mathbb{Z}_{2}$, is a linear map satisfying $f(V_{\alpha})\subseteq W_{\alpha+|f|}$ for $\alpha \in \mathbb{Z}_{2}$. In particular, if $|f| = \bar{0}$, then the homomorphism $f$ is called the homogeneous linear map of even degree. A Lie superalgebra homomorphism $f: L \rightarrow M$ is a  homogeneous linear map of even degree such that $f([x,y]) = [f(x), f(y)]$ holds for all $x, y \in L$.  If $I$ is an ideal of $L$, the quotient Lie superalgebra $L/I$ inherits a canonical Lie superalgebra structure such that the natural projection map becomes a homomorphism. The notions of {\it epimorphisms, isomorphisms,} and {\it automorphisms} have obvious meaning. Consider the descending central sequence $L^1 \supset L^2 \supset \cdots \supset L^k \supset \cdots $ of a Lie superalgebra $L = L_{\bar{0}} \oplus L_{\bar{1}}$ where $L^{1} = L, L^{k+1} = [L^{k}, L]$ for all $k\geq 1$. If $L^{k+1} = \{0\}$ for some $k+1$, with $L^{k} \neq 0$, then $L$ is called nilpotent with nilpotency class $k$. Note that the derived subalgebra is $[L,L]=L^2.$
\smallskip

We simply write $\dim L=(m\mid n)$ for the superdimension of Lie superalgebra $L$ throughout this article, where $\dim L_{\bar{0}} = m$ and $\dim L_{\bar{1}} = n$.  Additionally, $A(m \mid n)$ denotes an abelian Lie superalgebra with $\dim A=(m\mid n)$. In this article Heisenberg Lie superalgebra is a Lie superalgebra $L$ with $L^2=Z(L)$ and $\dim Z(L)=1$.  According to the homogeneous generator of $Z(L)$, Heisenberg Lie superalgebras can further split into even or odd Heisenberg Lie superalgebras \cite{MC2011}.  Further $L$ is called a generalized Heisenberg Lie superalgebra of rank $(r \mid s)$ if $L^2=Z(L)$ and $\dim Z(L)=(r \mid s)$. So a Heisenberg Lie super algebra is a generalized Heisenberg superalgebra of rank one, i.e., $(1 \mid 0)$ or $(0 \mid 1)$.
For more details on Heisenberg Lie superalgebra and their  multipliers see \cite{hasanpadtrip,Nayak2018,SN2018b,Padhandetec,hasanpadprad,p50,p54,p55}. Also, for more details on the Schur multiplier and capability of Lie algebras, see \cite{Alamian2008, Batten1993, Batten1996, BS1996, p1,Ellis1987,Ellis1991,Ellis1995, Ellis1996, HS1998,Kar1987, Moneyhun1994, nirmo2016, PMF2013}. Now we list some useful known results for further use.

\smallskip

\begin{thm}\label{th3.3}\cite[See Theorem 3.4]{Nayak2018}
\[\dim \mathcal{M}(A(m \mid n)) = \big(\frac{1}{2}(m^2+n^2+n-m)\mid mn \big).\]
\end{thm}

\begin{thm} \label{th3.4}\cite[See Theorem 4.2, 4.3]{Nayak2018}
Every Heisenberg Lie superalgebra with an even center has dimension $(2m+1 \mid n)$, is isomorphic to $H(m , n)=H_{\overline{0}}\oplus H_{\overline{1}}$, where
\[H_{\overline{0}}=\left\langle   x_{1},\ldots,x_{2m},z \mid [x_{i},x_{m+i}]=z,\ i=1,\ldots,m\right\rangle\]
and 
\[H_{\overline{1}}=\left\langle y_{1},\ldots,y_{n}\mid [y_{j}, y_{j}]=z,\  j=1,\ldots,n\right\rangle.\]
Further,
$$ 
\dim \mathcal{M}(H(m , n))=
\begin{cases}
 (2m^{2}-m+n(n+1)/2-1 \mid 2mn)\quad \mbox{if}\;m+n\geq 2\\
(0 \mid 0) \quad \mbox{if}\;m=0, n=1\\  
   (2 \mid 0)\quad \mbox{if}\;m=1, n=0.
\end{cases}
$$

\end{thm}

Similarly the multiplier and cover for Heisenberg Lie superalgebra of odd center is known.

\begin{thm}\label{th3.6}\cite[See Theorem 2.8]{SN2018b}
Every Heisenberg Lie superalgebra, with an odd center has dimension $(m \mid m+1)$, is isomorphic to $H_{m}=H_{\overline{0}}\oplus H_{\overline{1}}$, where
\[H_{m}=\left\langle x_{1},\ldots,x_{m} , y_{1},\ldots,y_{m},z \mid [x_{j},y_{j}]=z,   j=1,\ldots,m\right\rangle.\]
Further,
$$
\dim \mathcal{M}(H_{m})=
\begin{cases}
 (m^{2}\mid m^{2}-1)\quad \mbox{if}\;m\geq 2\\  
(1\mid 1) \quad \quad \mbox{if}\;m=1.  
  \end{cases}
$$
\end{thm}

For any two Lie superalgebras $H$ and $K$ the Lie superalgebra direct sum $H \oplus K$ is a Lie superalgebra with natural grading $(H \oplus K)_{\alpha}=H_{\alpha}\oplus K_{\alpha}$ where $\alpha \in \mathbb{Z}_2$. If $\mathcal{M}(H)$ and $\mathcal{M}(K)$ are known, then $\mathcal{M}(H \oplus K)$ is given by the following result.

\begin{thm}\label{th3.7}\cite[See Theorem 3.9]{Nayak2018}
For Lie superalgebras $H$ and $K$,
\[\mathcal{M}(H\oplus K)\cong \mathcal{M}(H)\oplus \mathcal{M}(K)\oplus (H/H^2\otimes K/K^2).\]
\end{thm}

A Lie superalgebra $L$ is capable
if there exists a Lie superalgebra $H$
such that $L \cong H/Z(H)$. The epicenter $Z^*(L)$  of $L$ is the smallest graded ideal in $L$ such that $L/Z^*(L)$ is capable \cite{Padhandetec}. Infact $Z^*(L)$ is a central ideal.

\begin{lemma}\label{lem2.6}\cite[Lemma 5.2]{Padhandetec}
A Lie superalgebra $L$ is capable if and only if $Z^*(L)= 0$.
\end{lemma}

\begin{lemma}\label{cor22}\cite[Theorem 5.8]{Padhandetec}
	$ N\subseteq Z^*(L)$ if and only if the natural map $\mathcal{M}(L) \longrightarrow \mathcal{M}(L/N)$ is a monomorphism.
\end{lemma}

\begin{theorem}\label{thm555}\cite[Theorem 6.7]{Padhandetec} 
	$H_m$ is capable if and only if $m = 1$.

\end{theorem}

\begin{theorem}\label{th4.22}\cite[ Theorem 6.3]{Padhandetec}
	$A(m \mid n)$ is capable if and only if $m=0, n=1$ or $m+n \geq 2$.
\end{theorem}

\begin{theorem}\label{th4.33}\cite[ Theorem 6.4]{Padhandetec}  
	$H(m, n)$ is capable if and only if $m=1, n=0$.
\end{theorem}

\section{ The Main Results}

In this section we define the partially capable Lie superalgebra, as a result we classify all capable nilpotent Lie superalgebras $L$ from \cite{Hegazi99, alv} with $\dim L \leq 5$. We organise all nilpotent Lie superalgebras from \cite{Hegazi99, alv} according to the derived algebras dimension.

\begin{theorem}
Let $L$ be a finite-dimensional nilpotent Lie superalgebra over $\mathbb{C} $. 
\begin{enumerate}
\item The following are the Lie superalgebras with $\dim L^2 = 1:$\\\\
$L^1_{(1,1)}=\{\left\langle 	x_1 \right\rangle \oplus \left\langle 	x_2 \right\rangle | ~[x_2,x_2]=x_1\}.$\\	
		$L^2_{(3,0)}=\{\left\langle 	x_1, x_2,x_3 \right\rangle \oplus \left\langle 	0 \right\rangle | ~[x_1,x_2]=x_3\}.$ \\	
		$(L^{1}_{(1,1)})'$ Derived from $L^1_{(1,1)} =\{\left\langle 	x_1, x_2 \right\rangle \oplus \left\langle 	x_3 \right\rangle | ~[x_3,x_3]=x_1, [x_1,x_2]=0\}$. \\	
		$L^3_{(1,2)}=\{\left\langle 	x_1 \right\rangle \oplus \left\langle 	x_2,x_3 \right\rangle | ~[x_1,x_2]=x_3\}.$\\	
		$L^4_{(1,2)}=\{\left\langle 	x_1 \right\rangle \oplus \left\langle 	x_2,x_3 \right\rangle | ~[x_2,x_2]=x_1, [x_3,x_3]=x_1\}.$ \\	
		$L^7_{(3,1)}=\{\left\langle 	x_1,x_2,x_3 \right\rangle \oplus \left\langle 	x_4 \right\rangle | ~[x_1,x_2]=x_3, [x_4,x_4]=x_3\}.$ \\	
		$L^{14}_{(1,3)}=\{\left\langle 	x_1 \right\rangle \oplus \left\langle 	x_2,x_3,x_4 \right\rangle | ~[x_2,x_2]=x_1, [x_3,x_3]=x_1],[x_4,x_4]=x_1\} .$ \\	
		$L^{16}_{(5,0)}=\{\left\langle 	x_1,x_2,x_3,x_4,x_5 \right\rangle \oplus \left\langle  0	 \right\rangle | ~[x_1,x_2]=x_5, [x_3,x_4]=x_5\}.$\\	
		$L^{25}_{(3,2)}=\{\left\langle 	x_1,x_2,x_3 \right\rangle \oplus \left\langle  x_4,x_5	 \right\rangle | ~[x_1,x_2]=x_3, [x_4,x_4]=x_3, [x_5,x_5]=x_3 \}.$\\
		$L^{26}_{(3,2)}=\{\left\langle 	x_1,x_2,x_3 \right\rangle \oplus \left\langle  x_4,x_5	 \right\rangle | ~[x_1,x_2]=x_3, [x_4,x_4]=x_3, [x_5,x_5]=-x_3\} .$\\				
		$L^{40}_{(1,4)}=\{\left\langle 	x_1 \right\rangle \oplus \left\langle  x_2,x_3,x_4,x_5	 \right\rangle | ~[x_2,x_2]=x_1, [x_3,x_3]=x_1, [x_4,x_4]=x_1, [x_5,x_5]=x_1\}.$\\		
		$L^{41}_{(1,4)}=\{\left\langle 	x_1 \right\rangle \oplus \left\langle  x_2,x_3,x_4,x_5	 \right\rangle | ~[x_2,x_2]=x_1, [x_3,x_3]=x_1, [x_4,x_4]=x_1, [x_5,x_5]=-x_1\}.$\\		
		$L^{42}_{(1,4)}=\{\left\langle 	x_1 \right\rangle \oplus \left\langle  x_2,x_3,x_4,x_5	 \right\rangle | ~[x_2,x_2]=x_1, [x_3,x_3]=x_1, [x_4,x_4]=-x_1, [x_5,x_5]=-x_1\}.$\\
		
\item The following are the Lie superalgebras with $\dim L^2 = 2:$\\\\	
$L^{6}_{(4,0)}=\{\left\langle 	x_1,x_2,x_3,x_4 \right\rangle \oplus \left\langle  0	 \right\rangle | ~[x_1,x_2]=x_3, [x_1,x_3]=x_4\}.$\\		
$L^{8}_{(2,2)}=\{\left\langle 	x_1,x_2 \right\rangle \oplus \left\langle  x_3,x_4	 \right\rangle | ~[x_1,x_3]=x_4, [x_3,x_3]=x_2\}.$\\		
$L^{9}_{(2,2)}=\{\left\langle 	x_1,x_2 \right\rangle \oplus \left\langle  x_3,x_4	\right\rangle | ~[x_3,x_3]=x_1, [x_4,x_4]=x_2, [x_3,x_4]=\frac{1}{2} (x_1+x_2)\}.$\\		
$L^{10}_{(2,2)}=\{\left\langle 	x_1,x_2 \right\rangle \oplus \left\langle  x_3,x_4	 \right\rangle | ~[x_3,x_3]=x_1, [x_4,x_4]=x_2\}.$\\		
$L^{11}_{(2,2)}=\{\left\langle 	x_1,x_2 \right\rangle \oplus \left\langle  x_3,x_4	 \right\rangle | ~[x_3,x_3]=x_1, [x_4,x_4]=x_2, [x_3,x_4]= x_1-x_2\}.$\\		
$L^{12}_{(2,2)}=\{\left\langle 	x_1,x_2 \right\rangle \oplus \left\langle  x_3,x_4	 \right\rangle | ~[x_3,x_3]=x_1, [x_4,x_4]=x_2, [x_3,x_4]= x_1\}.$\\		
$L^{13}_{(1,3)}=\{\left\langle 	x_1 \right\rangle \oplus \left\langle  x_2,x_3,x_4	 \right\rangle | ~[x_1,x_2]=x_3, [x_1,x_3]=x_4\}.$\\		
$L^{17}_{(5,0)}=\{\left\langle 	x_1,x_2,x_3,x_4,x_5 \right\rangle \oplus \left\langle  0	 \right\rangle | ~[x_1,x_2]=x_4, [x_1,x_3]=x_5\}.$\\		
$L^{18}_{(5,0)}=\{\left\langle 	x_1,x_2,x_3,x_4,x_5 \right\rangle \oplus \left\langle  0	 \right\rangle | ~[x_1,x_2]=x_3, [x_1,x_3]=x_4, [x_2,x_5]=x_4\}.$\\		
$L^{22}_{(4,1)}=\{\left\langle 	x_1,x_2,x_3,x_4 \right\rangle \oplus \left\langle  x_5	 \right\rangle | ~[x_4,x_2]=x_1, [x_4,x_3]=x_3, [x_5,x_5]=x_1\}.$\\		
$L^{23}_{(4,1)}=\{\left\langle 	x_1,x_2,x_3,x_4 \right\rangle \oplus \left\langle  x_5	 \right\rangle | ~[x_1,x_2]=x_3, [x_4,x_2]=x_1, [x_5,x_5]=x_3\}.$\\		
$L^{24}_{(3,2)}=\{\left\langle 	x_1,x_2,x_3 \right\rangle \oplus \left\langle  x_4,x_5	 \right\rangle | ~[x_1,x_2]=x_3, [x_1,x_5]=x_4\}.$\\
$L^{27}_{(3,2)}=\{\left\langle 	x_1,x_2,x_3 \right\rangle \oplus \left\langle  x_4,x_5	 \right\rangle | ~[x_1,x_2]=x_3, [x_1,x_5]=x_4, [x_5,x_5]=x_3\} .$\\		$L^{28}_{(2,3)}=\{\left\langle 	x_1,x_2 \right\rangle \oplus \left\langle  x_3,x_4,x_5	 \right\rangle | ~[x_1,x_4]=x_3, [x_4,x_4]=x_2, [x_5,x_5]=x_2\}.$\\
$L^{29}_{(2,3)}=\{\left\langle 	x_1,x_2 \right\rangle \oplus \left\langle  x_3,x_4,x_5	 \right\rangle | ~[x_1,x_4]=x_3, [x_4,x_4]=x_2, [x_5,x_5]=-x_2\}.$\\	$L^{30}_{(2,3)}=\{\left\langle 	x_1,x_2 \right\rangle \oplus \left\langle  x_3,x_4,x_5	 \right\rangle | ~[x_3,x_3]=x_1, [x_4,x_4]=x_2, [x_5,x_5]=x_1+x_2\}.$\\	$L^{31}_{(2,3)}=\{\left\langle 	x_1,x_2 \right\rangle \oplus \left\langle  x_3,x_4,x_5	 \right\rangle | ~[x_3,x_3]=x_1, [x_4,x_4]=x_2, [x_5,x_5]=-(x_1+x_2)\}.$\\
$L^{32}_{(2,3)}=\{\left\langle 	x_1,x_2 \right\rangle \oplus \left\langle  x_3,x_4,x_5	 \right\rangle | ~[x_3,x_3]=x_1, [x_4,x_4]=x_2, [x_5,x_5]=x_1-x_2\}.$\\	
$L^{33}_{(2,3)}=\{\left\langle 	x_1,x_2 \right\rangle \oplus \left\langle  x_3,x_4,x_5	 \right\rangle | ~[x_3,x_3]=x_1, [x_4,x_4]=x_2, [x_3,x_5]=x_2\}.$\\
$L^{34}_{(2,3)}=\{\left\langle 	x_1,x_2 \right\rangle \oplus \left\langle  x_3,x_4,x_5	 \right\rangle | ~[x_3,x_3]=x_1, [x_4,x_4]=x_2, [x_3,x_5]=x_1+x_2\}.$\\
$L^{35}_{(2,3)}=\{\left\langle 	x_1,x_2 \right\rangle \oplus \left\langle  x_3,x_4,x_5	 \right\rangle | ~[x_3,x_3]=x_1, [x_4,x_4]=x_2, [x_3,x_5]=x_1-x_2\}.$\\	
$L^{36}_{(2,3)}=\{\left\langle 	x_1,x_2 \right\rangle \oplus \left\langle  x_3,x_4,x_5	 \right\rangle | ~[x_3,x_3]=x_1, [x_3,x_5]=x_2, [x_4,x_5]=x_1\}.$\\
$L^{37}_{(2,3)}=\{\left\langle 	x_1,x_2 \right\rangle \oplus \left\langle  x_3,x_4,x_5	 \right\rangle | ~[x_3,x_4]=x_1, [x_4,x_5]=x_2\}.$\\
$L^{38}_{(1,4)}=\{\left\langle 	x_1 \right\rangle \oplus \left\langle  x_2,x_3,x_4,x_5	 \right\rangle | ~[x_1,x_3]=x_2, [x_1,x_5]=x_4 \}.$\\
$L^{43}_{(2,3)}=\{\left\langle 	x_1,x_2 \right\rangle \oplus \left\langle  x_3,x_4,x_5	 \right\rangle | ~[x_1,x_5]=x_3, [x_4,x_5]=x_2 \}.$\\

\item The following are the Lie superalgebras with $\dim L^2 = 3:$\\\\
$L^{19}_{(5,0)}=\{\left\langle 	x_1,x_2,x_3,x_4 \right\rangle \oplus \left\langle  x_5	 \right\rangle | ~[x_1,x_2]=x_3, [x_1,x_3]=x_4, 
		[x_2,x_3]=x_5\}.$\\
$L^{20}_{(5,0)}=\{\left\langle 	x_1,x_2,x_3,x_4 \right\rangle \oplus \left\langle  x_5	 \right\rangle | ~[x_1,x_2]=x_3, [x_1,x_3]=x_4, 
		[x_1,x_4]=x_5.\}$\\
$L^{21}_{(5,0)}=\{\left\langle 	x_1,x_2,x_3,x_4 \right\rangle \oplus \left\langle  x_5	 \right\rangle | ~[x_1,x_2]=x_3, [x_1,x_3]=x_4, 
		[x_1,x_4]=x_5=[x_2,x_3]\}.$\\
$L^{39}_{(1,4)}=\{\left\langle 	x_1 \right\rangle \oplus \left\langle  x_2,x_3,x_4,x_5	 \right\rangle | ~[x_1,x_2]=x_3, [x_1,x_3]=x_4, 
[x_1,x_4]=x_5=[x_3,x_4]\}.$\\
$L^{44}_{(2,3)}=\{\left\langle 	x_1,x_2 \right\rangle \oplus \left\langle  x_3,x_4,x_5	 \right\rangle | ~[x_1,x_5]=x_3, [x_2,x_4]=x_3, [x_4,x_5]=-x_1, [x_5,x_5]=2x_2 \}.$\\
$L^{45}_{(2,3)}=\{\left\langle 	x_1,x_2 \right\rangle \oplus \left\langle  x_3,x_4,x_5	 \right\rangle | ~[x_1,x_4]=x_3, [x_1,x_5]=x_4, [x_5,x_5]=x_2 \}.$\\
$L^{46}_{(2,3)}=\{\left\langle 	x_1,x_2 \right\rangle \oplus \left\langle  x_3,x_4,x_5	 \right\rangle | ~[x_1,x_4]=x_3, [x_1,x_5]=x_4, [x_3,x_5]=-x_2, [x_4,x_4]=x_2 \}.$\\
\end{enumerate}
\end{theorem}
Consider the Heisenberg Lie superalgebra $H_2$ with a homogeneous basis $\{x_1, x_2; y_1,y_2,z\}$ and the only non-vanishing brackets are $[x_1, y_1]=z=[x_2, y_2]$. By Theorem $2.7$ it is not capable but the even part of this Lie supealgebra is capable Lie algebra as even part is isomorphic to $H(1, 0)/Z(H)$. This leads to following definition.

\begin{definition}
We call a Lie superalgebra $L = L_{\bar{0}} \oplus L_{\bar{1}}$ partially capable if the even part $L_{\bar{0}}$ is a capable Lie algebra.
\end{definition}

\begin{theorem}
Let $L = L_{\bar{0}} \oplus L_{\bar{1}}$ be a Lie superalgebra. Then $Z^*(L_{\bar{0}}) \subseteq Z^*(L)$.
\end{theorem}

\begin{proof}
Consider the two homomorphism, $\phi : L_{\bar{0}}/Z^*(L_{\bar{0}}) \longrightarrow L$ which is given by $\phi(l + Z^*(L_{\bar{0}})) = l$, and the canonical projection $ \pi: L\longrightarrow L/Z^*(L)$. The composition map $\psi= \pi \phi : L_{\bar{0}}/Z^*(L_{\bar{0}}) \longrightarrow L/Z^*(L)$ defined by $\psi(l + Z^*(L_{\bar{0}})) = l + Z^*(L)$ is also a homomorphism. Thus $Z^*(L_{\bar{0}}) \subseteq Z^*(L)$.
\end{proof}

The following is the immediate consequence of the above theorem.

\begin{corollary}
If $L = L_{\bar{0}} \oplus L_{\bar{1}}$ is a capable Lie superalgebra, then the even part is also capable.
\end{corollary}

\begin{remark}
Every capable Lie superalgebra $L$ is partially capable but converse is not true. However if $L$ is not partially capable then $L$ is not capable.
\end{remark}
From now onwards $L$ denote a nilpotent Lie superalgebra with $\dim L \leq 5$.

\subsection{The Capability of $L$ with $\dim L^2=1$}
  
In this case it has been proved that $L$ is direct sum of Heisenberg Lie superalgebra and abelian Lie superalgebra. Also all capable Lie superalgebras having derived subalgebras dimension at most one have been classified  (see \cite{Padhandetec}).

\begin{proposition}\label{th4.44}\cite[Proposition 3.4]{SN2018b}\cite[Theorem 6.9]{Padhandetec} \label{th5a}
Let $L$ be a nilpotent Lie superalgebra of dimension $(k \mid l)$ with $\dim L^2=(r \mid s)$, where $r+s=1$. If $r=1, s=0$ then $L \cong H(m,n)\oplus A(k-2m-1 \mid l-n)$ for $m+n\geq 1$. If $r=0, s=1$ then $L \cong H_{m} \oplus A(k-m \mid l-m-1)$. Moreover, $L$ is capable if and only if either $L \cong H(1 , 0)\oplus A(k-3 \mid l)$ or $L \cong H_{1}\oplus A(k-1 \mid l-2)$.
\end{proposition}

\begin{proposition} \label{l1}
$L^{1}_{(1,1)},L^{3}_{(1,2)},L^{4}_{(1,2)},L^{14}_{(1,3)},L^{40}_{(1,4)},L^{41}_{(1,4)},L^{42}_{(1,4)}$ are non-capable Lie superalgebras. 
\end{proposition}

\begin{proof}
 Even part of all these Lie superalgebras are $A(1\mid0)$. So by Theorem \ref{th4.22}, all these superalgebras are not partially capable.
\end{proof}

\begin{proposition}\label{l2}
$L^{7}_{(3,1)},~L^{16}_{(5,0)},~L^{25}_{(3,2)},(L^{1}_{(1,1)})^2$ are non-capable Lie superalgebras.
\end{proposition}

\begin{proof}
We have $L^{7}_{(3,1)}\cong H(1,  1),~ L^{16}_{(5,0)}\cong H(2 , 0)$ and $L^{25}_{(3,2)}\cong H(1, 2)$. Now, it follows from Theorem \ref{th4.33} that $L^{7}_{(3,1)},~L^{16}_{(5,0)},~L^{25}_{(3,2)}$ are non-capable. Finally as  $(L^{1}_{(1,1)})^{2}\cong H(0 , 1)\oplus A(1 \mid 0)$, using Proposition \ref{th4.44} $(L^{1}_{(1,1)})^{2}$ is non-capable. 
\end{proof}

From the above, we can draw the following conclusion.

\begin{theorem}
Let $L$ be a nilpotent Lie superalgebra with $\dim L \leq 5$ and $\dim L^2 =1$. Then $L$ is capable if and only if $L \cong L^{2}_{(3,0)} = H(1 \mid 0)$.
\end{theorem}

\begin{proof}
The proof follows from Theorem \ref{th4.33}, Propositio \ref{l1} and Proposition \ref{l2}.
\end{proof}

\subsection{The Capability $L$ with $\dim L^2=2$.}

The following results will be useful in this section. 
\begin{proposition}\label{lem3.11}\cite[Proposition 2.6]{nirmo2016}
Let $L$ be a finite dimensional capable nilpotent Lie algebra with nilpotency class two. Then
	$5 \leq \dim L \leq 7.$
\end{proposition}

\begin{proposition}\label{lem3.12}\cite[Corollary 3.11]{p50}
Let $L$ be a Lie superalgebra of dimension $(m  \mid n)$ with nilpotency class two such that $\dim(L/Z(L)) = (r  \mid s)$ and $\dim(L^2) =\frac{1}{2}[(r + s)^2 + (s - r)]$, then $L$ is capable.
\end{proposition}

In \cite{Nayak2018} Nayak calculated the dimension of Schur multiplier of some nilpotent Lie superalgebras.

\begin{proposition}\label{lem3.15} \cite[Theorems  3.2, 3.6]{Nayak2018}  
$\dim\mathcal{M}(L^{24}_{(3,2)})=6$, $\dim\mathcal{M}(L^l_{(2,2)})=1$  (with $l =
9, ~10,~ 11,~ 12)$  and $\dim\mathcal{M}(L^k_{(2,3)})=4$  (with $k = 28,~ 30, ~31, ~32, ~33, ~34,~ 35)$. 
\end{proposition}

The dimensions of Schur multipliers of the nilpotent Lie superalgebras $L^{26}_{(3,2)}, L^{27}_{(3,2)}, L^{29}_{(2,3)}, L^{36}_{(2,3)}, L^{37}_{(2,3)}$ are not computed in \cite{Nayak2018}. We give the rest of the computations in the following lemma by using the method described by Hardy and Stitzinger in \cite{HS1998}.

\begin{proposition}
$\dim\mathcal{M}(L^{26}_{(3,2)})=7, \dim\mathcal{M}(L^{27}_{(3,2)})=\dim\mathcal{M}(L^{29}_{(2,3)})=\dim\mathcal{M}(L^{36}_{(2,3)}) =\dim\mathcal{M}(L^{44}_{(2,3)})=4,\dim\mathcal{M}(L^{37}_{(2,3)})=\dim\mathcal{M}(L^{43}_{(2,3)})=5,\dim\mathcal{M}(L^{45}_{(2,3)})=\dim\mathcal{M}(L^{46}_{(2,3)})=3.$
\end{proposition}
\begin{proof}
Consider $$L^{27}_{(3,2)}=\{\left\langle 	x_1,x_2,x_3 \right\rangle \oplus \left\langle  x_4,x_5	 \right\rangle | ~[x_1,x_2]=x_3, [x_1,x_5]=x_4, [x_5,x_5]=x_3\} .$$ To compute the multiplier of $L^{27}_{(3,2)},$ we start with the following setting 
\begin{center}
\begin{tabular}{ c c c }
		$[x_1,x_2]=x_3+m_1$ & $[x_1,x_3]=m_2$ & $[x_1,x_4]=m_3$ \\ 
		$[x_1,x_5]=x_4+m_4$ & $[x_2,x_3]=m_5$ & $[x_2,x_4]=m_6$ \\  
		$[x_2,x_5]=m_7$     & $[x_3,x_4]=m_8$ & $[x_3,x_5]=m_9$ \\
		$[x_4,x_5]=m_{10}$  & $[x_4,x_4]=m_{11}$ & $[x_5,x_5]=x_3+m_{12}$.
\end{tabular}
\end{center}
Putting $x'_3=x_3+m_1,~ x'_4=x_4+m_4$, then  $m_1=m_4=0$. By applying the graded Jacobi identity on all possible triples, we have
\begin{align*}
m_2-2m_{10}&=(-1)^{|x_1||x_5|}[x_1,[x_5,x_5]]+(-1)^{|x_1||x_5|}[x_5,[x_5,x_1]] +(-1)^{|x_5||x_5|}[x_5,[x_1,x_5]]=0 \\ m_5&=(-1)^{|x_2||x_5|}[x_2,[x_5,x_5]]+(-1)^{|x_2||x_5|}[x_5,[x_5,x_2]] +(-1)^{|x_5||x_5|}[x_5,[x_2,x_5]]=0 \\
-m_6-m_{9}&=(-1)^{|x_1||x_5|}[x_1,[x_2,x_5]]+(-1)^{|x_1||x_2|}[x_2,[x_5,x_1]] +(-1)^{|x_5||x_2|}[x_5,[x_1,x_2]]=0 \\
	-m_8&=(-1)^{|x_1||x_5|}[x_1,[x_3,x_5]]+(-1)^{|x_3||x_1|}[x_3,[x_5,x_1]] +(-1)^{|x_5||x_3|}[x_5,[x_1,x_3]]=0 \\
	-m_{11}&=(-1)^{|x_1||x_5|}[x_1,[x_4,x_5]]+(-1)^{|x_4||x_1|}[x_4,[x_5,x_1]] +(-1)^{|x_4||x_5|}[x_5,[x_1,x_4]]=0 \\
	m_9&=[x_5,[x_5,x_5]]=0.
\end{align*}
Hence $\mathcal{M}(L^{27}_{(3,2)})=\left\langle m_2,m_3,m_7,m_{12} \right\rangle $  and $\dim\mathcal{M}(L^{27}_{(3,2)})=4$. Similarly the rest can be proven.
\end{proof}

\begin{proposition}\cite[Proposition 3.1]{p50} Let $L$ be a finite dimensional nilpotent Lie superalgebra of nilpotency class two. Then $L = H \oplus K$ and $Z^*(L) = Z^*(H)$, where $K$ is abelian and $H$ is a generalized Heisenberg Lie superalgebra.
\end{proposition}

\begin{lemma}\label{lem3.10}
If $L$ is finite dimensional, then $ N\subseteq Z^*(L)$  if and only if $\dim \mathcal{M}(L/N) = \dim \mathcal{M}(L) + \dim (N \cap L^2)$.
\end{lemma}
\begin{proof}
As we know $\mathcal{M}(L) \xrightarrow{\sigma}\mathcal{M}(L/N) \xrightarrow{\varphi} N\cap L^2\longrightarrow 0$ is exact sequence, thus 
$\mathcal{M}(L/N)\cong ker \varphi \oplus Im \varphi = Im \sigma \oplus N \cap L^2  $. We can observe that if  $\sigma$ is a monomorphism, then $\dim \mathcal{M}(L/N) = \dim \mathcal{M}(L) + \dim (N \cap L^2)$. Therefore, the proof follows from Lemma  \ref{cor22}.
\end{proof}

\begin{corollary}\label{cor1}
Let $L = H(t,q) \oplus H(m,n)$. If $t+q \geq2$, or $t=0, q=1$ then $L$ is non-capable for all $m+n\geq 1$.
\end{corollary}
\begin{proof}
If we use Proposition \ref{th3.7}, then
\[\mathcal{M}(L)= \mathcal{M}(H(t,q))\oplus \mathcal{M}(H(m,n))\oplus (H(t,q)/H^2(t,q)\otimes H(m,n)/H^2(m,n)).\]
Now,
\[\mathcal{M}(L/H^2(t,q))= \mathcal{M}(H(t,q)/H^2(t,q))\oplus \mathcal{M}(H(m,n))\oplus (H(t,q)/H^2(t,q)\otimes H(m,n)/H^2(m,n)).\]
Now from Theorem \ref{th3.3} and Theorem \ref{th3.4}, we have 
\[	\mathcal{M}(L)=\mathcal{M}(L/H^2(t,q))-(1\mid 0).\]
Thus, using Proposition \ref{lem3.10}, $H^2(t,q)\subseteq 	Z^*(L)$, and $L$ is non-capable.    
\end{proof}

\begin{example}
Let $L=L^{10}_{(2,2)}=\{\left\langle 	x_1,x_2 \right\rangle \oplus \left\langle  x_3,x_4	 \right\rangle | ~[x_3,x_3]=x_1, [x_4,x_4]=x_2\}.$ Clearly $L\cong H(0, 1) \oplus H(0, 1)$ and hence $L$ is non-capable by corollary \ref{cor1}.
\end{example}
\begin{thm}
$L^{6}_{(4,0)},L^{13}_{(1,3)},L^{17}_{(5,0)},L^{18}_{(5,0)},L^{22}_{(4,1)},L^{23}_{(4,1)},L^{38}_{(1,4)}$ are non-capable Lie superalgebras. 
\end{thm}
\begin{proof}
One can observe that $L^{6}_{(4,0)},L^{22}_{(4,1)},L^{23}_{(4,1)}$ are not partially capable by Proposition \ref{lem3.11}. Fom Theorem \ref{th4.22}, $L^{38}_{(1,4)}, L^{13}_{(1,3)}$ are not partially capable . 
\end{proof}

\begin{theorem}\label{th3.16}
$L^{9}_{(2,2)},  L^{11}_{(2,2)}, L^{12}_{(2,2)}$ are non-capable Lie superalgebras.
\end{theorem}

\begin{proof}
Consider $L=L^{9}_{(2,2)}=\{\left\langle 	x_1,x_2 \right\rangle \oplus \left\langle  x_3,x_4	 \right\rangle | ~[x_3,x_3]=x_1, [x_4,x_4]=x_2, [x_3,x_4]=\frac{1}{2} (x_1+x_2)\}$, then $Z(L)=L^2=\left\langle 	x_1,x_2 \right\rangle$. Consider $N=L^2$, then by Theorem \ref{th3.3}, $\dim\mathcal{M}(L/N)=\dim\mathcal{M}(A(0 \mid 2))= (3 \mid 0)=3$. From Proposition\ref{lem3.15}, $\dim\mathcal{M}(L)=1$, then $3=\dim \mathcal{M}(L/N) = \dim \mathcal{M}(L) + \dim N \cap L^2=2+1$, so by using Lemma \ref{lem3.10}, we have $L^2 \subseteq Z^*(L)$. Hence $L=L^{9}_{(2,2)}$ is non-capable. Similarly, we one can check that $L^{11}_{(2,2)}, L^{12}_{(2,2)}$ are non-capable.
\end{proof}

\begin{theorem}
$L^{30}_{(2,3)}, L^{31}_{(2,3)}, L^{32}_{(2,3)}, L^{33}_{(2,3)}, L^{34}_{(2,3)},$ $L^{26}_{(3,2)}, L^{27}_{(3,2)}, L^{29}_{(2,3)}, L^{35}_{(2,3)}$  and $L^{36}_{(2,3)}$ are non-capable Lie superalgebras.
\end{theorem}
\begin{proof}
The proof is similar to Theorem \ref{th3.16}. 
\end{proof}

\begin{theorem}\label{them3.20}
The only capable Lie superalgebras with two dimension derived subalgebras are
$L^8_{(2,2)}, L^{24}_{(3,2)}, L^{28}_{(2,3)}, L^{17}_{(5,0)}, L^{18}_{(5,0)},L^{37}_{(2,3)},L^{43}_{(2,3)}$.
\end{theorem}
\begin{proof}
As $L^{8}_{(2,2)}=\{\left\langle 	x_1,x_2 \right\rangle \oplus \left\langle  x_3,x_4	 \right\rangle | ~[x_1,x_3]=x_4, [x_3,x_3]=x_2\}$ is nilpotent Lie superalgebra of class two with $\dim (L/Z(L))=(1\mid 1)$ and  $ \frac{1}{2}[(1 + 1)^2 + (1 - 1)]=2=\dim (L^2)$, then $L^{8}_{(2,2)}$ is capable by Proposition \ref{lem3.12}. For $L=L^{24}_{(3,2)}=\{\left\langle 	x_1,x_2,x_3 \right\rangle \oplus \left\langle  x_4,x_5	 \right\rangle | ~[x_1,x_2]=x_3, [x_1,x_5]=x_4\},$ we have $Z(L)=L^2=\{\left\langle x_3,x_4 \right\rangle\}$, and by Proposition \ref{lem3.15}, $\dim\mathcal{M}(L)= 6.$ Now,  let $N$ be a $(1 \mid 0)$-dimensional central graded ideal generated by $x_3$. Then using Theorems \ref{th3.3},\ref{th3.6} and \ref{th3.7}, we have 
\begin{align*}
\dim\mathcal{M}(L/N)   &=\dim\mathcal{M}(H_1\oplus A(1\mid 0))\\
		& =\dim\mathcal{M}(H_1)+\dim\mathcal{M} (A(1\mid 0))+\dim (H_1/H^2_1\otimes A(1\mid 0))\\
		&=(1\mid 1)+(0\mid 0)+ (1\mid 1)\times(1 \mid 0)\\
		&=4.
\end{align*}
This shows that $\dim \mathcal{M}(L/N)=4 \neq \dim \mathcal{M}(L) + \dim (N \cap L^2)=6+1$. Hence from Lemma \ref{lem3.10}, $N\nsubseteq Z^*(L)$. Similarly, if we consider a $(0 \mid 1)$-dimensional central graded ideal $K$ which is generated by $x_4$, then $\dim \mathcal{M}(L/K)=5 \neq \dim \mathcal{M}(L) + \dim (K\cap L^2)=6+1$. Thus $K\nsubseteq Z^*(L)$. Therefore, by Lemma \ref{lem2.6},  $L=L^{24}_{(3,2)}$ is capable. Using same technique one can check that $L^{28}_{(2,3)},~L^{37}_{(2,3)},L^{43}_{(2,3)}$ are capable. Finally, from \cite[Theorems 2.6, 1.1]{SNR2022}, we have $L^{17}_{(5,0)}=L_{5,8},~L^{18}_{(5,0)}=L_{5,5}$ are capable Lie algebras.
\end{proof}
\begin{remark}
The technique used to prove $ L^8_{(2,2)}$ is capable, does  not work for $L^{24}_{(3,2)}, L^{28}_{(2,3)},L^{37}_{(2,3)}$, so we use a different one.
\end{remark}

\subsection{The Capability of all nilpotent Lie superalgebras $L$ with $\dim L \leq 5$ and $\dim L^2 \leq 3.$}
\begin{lemma}
$L^{39}_{(1,4)}$ is non-capable Lie superalgebra.
\end{lemma}
\begin{proof}
As the even part of $L^{39}_{(1,4)}$ is $A(1\mid 0)$, then by Theorem \ref{th4.22}, $L^{39}_{(1,4)}$  is not partially capable.
\end{proof}
\begin{lemma}
	$L^{44}_{(2,3)},L^{45}_{(2,3)}, ~L^{46}_{(2,3)}$ are capable Lie superalgebras.
\end{lemma}

\begin{proof}
Using Proposition \ref{lem3.15}, and proceeding as inTheorem \ref{them3.20}, for every central  ideal $N$ of $L^{i}_{(2,3)}$ we have, 
		$\dim\mathcal{M}(L^{i}_{(2,3)}/N) \neq\dim \mathcal{M}(L^{i}_{(2,3)}) + \dim N \cap L^{i^2}_{(2,3)},$ for $i=44,45,46.$ Thus, 	$L^{45}_{(2,3)}, ~L^{46}_{(2,3)}$ are capable Lie superalgebras..
\end{proof}

\begin{lemma}

$L^{19}_{(5,0)}, ~L^{20}_{(5,0)},~ L^{21}_{(5,0)}$ are capable Lie superalgebras.
\end{lemma}

\begin{proof}
From \cite[Theorems 2.6, 1.1]{SNR2022},  $L^{19}_{(5,0)}=L_{5,9},~L^{20}_{(5,0)}=L_{5,7},~ L^{21}_{(5,0)}=L_{5,6} $ are capable Lie algebras.
\end{proof}

Finally following is the list all capable nilpotent Lie superalgebras $L$ with $\dim L \leq 5.$

\begin{theorem}
The only capable nilpotent Lie superalgebras with dimension at most 5 are
$$L^2_{(3 \mid 0)}, ~L^8_{(2,2)}, ~ L^{17}_{(5,0)},~ L^{18}_{(5,0)}, ~ L^{19}_{(5,0)},~ L^{20}_{(5,0)}, ~L^{21}_{(5,0)}, ~L^{24}_{(3,2)}, ~L^{28}_{(2,3)},~L^{37}_{(2,3)},~L^{43}_{(2,3)},~L^{44}_{(2,3)},~L^{45}_{(2,3)},~L^{46}_{(2,3)}.$$
\end{theorem}

\end{document}